\documentclass[a4paper,11pt,twoside]{amsart}
\usepackage[english]{babel}
\usepackage[utf8]{inputenc}

\usepackage[a4paper,inner=3cm,outer=3cm,top=4cm,bottom=4cm,pdftex]{geometry}
\usepackage{fancyhdr}
\usepackage{color}
\usepackage{bold-extra}
\usepackage{ mathrsfs }

\usepackage{comment}
\usepackage{graphics}
\usepackage{aliascnt}
\usepackage[pdftex,citecolor=green,linkcolor=red]{hyperref}

\usepackage{amsmath}
\usepackage{amsfonts}
\usepackage{amssymb}
\usepackage{amsthm}
\usepackage{comment}
\usepackage{mathtools}
\usepackage{enumerate}
\usepackage{enumitem} 

\newtheorem{theorem}{Theorem}[section]

\newtheorem{lemma}[theorem]{Lemma}
\newtheorem{proposition}[theorem]{Proposition}
\theoremstyle{definition}

\numberwithin{equation}{section}

\newcommand\R{\mathbb{R}}
\newcommand\Z{\mathbb{Z}}

\newcommand\C{\mathbb{C}}

\newcommand{\CP}{\mathcal{P}}
\newcommand{\vdelta}{\underline{\delta}}
\newcommand{\wh}{\widehat}

\begin{document}
\title{Density Ternary Goldbach for primes in a fixed residue class}

\author[Alsetri]{Ali Alsetri}
\address{Department of Mathematics, University of Kentucky\\
715 Patterson Office Tower\\
Lexington, KY 40506\\
USA}
\email{alialsetri@uky.edu}



\begin{abstract}
We prove that if $A$ is a subset of those primes which are congruent to $1 \pmod{3}$ such that the relative density of $A$ in this residue class is larger than $\frac{1}{2},$ then every sufficiently large odd integer $n$ which satisfies $n \equiv 0 \pmod{3}$ can be written as a sum of three primes from $A.$ Moreover the threshold of $\frac{1}{2}$ for the relative density is best possible.

\end{abstract}
\maketitle

\section{Introduction}
This article falls under the general theme of representations of integers as sums of primes. The famous binary Goldbach conjecture states that every even integer $n \geq 4$ can be written as the sum of two primes. While the binary Goldbach problem is still out of reach, the ternary Goldbach problem which studies representations of odd integers as sums of three primes has now been solved. Vinogradov \cite{Vinogradov} showed in $1937$ that every sufficiently large odd integer is a sum of three primes, and in $2015$ Helfgott \cite{Helfgott} showed this holds for every odd integer $n \geq 7.$

There has recently been an effort to obtain density versions of the aforementioned Goldbach-type results. In $2005$ Green \cite{Green} established a version of Roth's theorem for the primes; he proved that any subset of the primes with positive relative density contains a three term arithmetic progression. In this paper he developed what has become known as a transference principle, which allows one to transfer certain results which hold for dense subsets in the integers to relatively dense subsets of the primes. Variants of Green's transference principle have been developed to allow for density versions of Goldbach-type problems. Given a subset $A \subset \mathbb{P}$ of the primes define its lower density in the primes to be $$\underline{\delta_{\mathbb{P}}}(A) = \liminf_{N \rightarrow \infty} \frac{|A\cap \{1,2,\dots,N\}|}{|\mathbb{P} \cap \{1,2,\dots,N\}|}.$$ Li and Pan \cite{LiPan} developed a variant of Green's transference principle suitable for ternary Goldbach-type problems and used it to prove a density version of Vinogradov's theorem. Namely they showed that if $A,B,C \subset \mathbb{P}$ such that $\underline{\delta_{\mathbb{P}}}(A) + \underline{\delta_{\mathbb{P}}}(B) + \underline{\delta_{\mathbb{P}}}(C) > 2,$ then every sufficiently large odd integer $n$ can be written as a sum $n = a + b + c$ with $a \in A,b \in B$ and $c \in C.$ Shao \cite{Shao}, using Li and Pan's variant of the transference principle, improved their result in the case when the three primes belong to the same subset of the primes. He proved that if $A \subset \mathbb{P}$ such that $\underline{\delta_{\mathbb{P}}}(A) > \frac{5}{8}$ then every sufficiently large odd integer $n$ can be written as a sum  $n = a_1 + a_2 + a_3$ with $a_1,a_2,a_3 \in A.$ Moreover he showed that the threshold $\frac{5}{8}$ was sharp. Shen \cite{Shen} showed Shao's result continues to hold for three sets $A,B,C \subset \mathbb{P}$ where the lower density of one is greater than $\frac{5}{8}$ and the lower densities of the other two are at least $\frac{5}{8}.$ For other applications of the transference principle to Vinogradov type results for sparse subsets of the primes see \cite{Shao-Kaisa-Joni,Shao-Kaisa, Grimmelt}. For applications of the transference principle to binary Goldbach type problems see \cite{ChipeniukHamel,Joni,Matomaki,Alsetri-Shao}, and for a general survey of the transference principle see \cite{Prendiville}.

In this paper we prove a density version of a variant of Vinogradov's theorem in which the primes in our subset belong to the residue class $1 \pmod{3}.$ Sums of three such primes must be multiples of $3$, and so we seek to represent every sufficiently large odd integer $n$ such that $n \equiv 0 \pmod{3}$ as a sum of three primes from our subset. We will now precisely state our main result.
\begin{theorem}\label{density-prime-1/2}
Let $\mathcal{P}_{1 , 3}$ denote the primes congruent to $1$ modulo $3$, and let $A \subset \mathcal{P}_{1,3}$ satisfy $\underline{\delta}(A) > \frac{1}{2}$ where $\underline{\delta}(A) = \liminf_{N \rightarrow \infty}{\frac{|A \cap [1,N]|}{|\mathcal{P}_{1,3} \cap [1,N]|}}$. Then $A+A+A$ contains every sufficiently large odd integer $n$ which satisfies $n \equiv 0 \pmod{3}$.
\end{theorem}

Our strategy in this paper follows the strategy adopted by Shao \cite{Shao}. To facilitate his use of the transference principle he develops a local result, Proposition $1.4$ in \cite{Shao}, which is the key technical tool in his paper. It implies that if $m$ is an odd squarefree integer and $A \subset \Z_m^*$, then $A+A+A$ contains all residue classes modulo $m$ provided $|A| > \frac{5}{8}\phi(m).$ This is Corollary $1.5$ in \cite{Shao}, it follows immediately from the more quantitative Proposition $1.4$ preceding it.  We will develop an analog of his local result which will allow us to deduce our main theorem using a Fourier analytic transference principle. We will now state our local result.

\begin{theorem}\label{local-result}

Let $m$ be an odd squarefree integer which satisfies $m \equiv 0 \pmod{3}.$ Let $f:\Z_{m}^* \rightarrow [0,1]$ be a function supported on those reduced residue classes which are $1 \pmod3$ which satisfies $$\sum_{a \in \Z_m^*}f(a) > \frac{\varphi(m)}{4}.$$ Then for any $x \in \Z_m$ for which $x \equiv 0 \pmod 3$, there exist $a_1,a_2,a_3 \in \Z_m^*$ such that $$f(a_i) > 0, f(a_1) + f(a_2) + f(a_3) > \frac{3}{2}.$$

\end{theorem}

This result is sharp. Let $A = \{1,7\} \subset \Z_{15}^*$, then the residue class $12 \pmod{15} \notin A+A+A.$ Therefore the corresponding function $f:\Z_{15}^* \rightarrow [0,1]$ which is defined by $$f(x)=
\begin{cases}
 1&\text{if }  x \in A \\
      0 &\text{otherwise}
\end{cases}
$$
has its average $\sum_{x \in \Z_{15}^*}f(x) = \frac{\phi(15)}{4}$ but the conclusion of the theorem fails upon taking $x = 12\pmod{15}.$

This article is organized as follows. In Section \ref{sec2} we prove some key lemmas which are needed for the proof of our local result \ref{local-result}. In Section \ref{sec3} we prove our local result \ref{local-result} which we deduce from the two propositions \ref{local-prop1} and \ref{local-prop2}. Then in Section \ref{sec4} we include the variant of the transference principle which we need in the form used by Shao in \cite{Shao}. Finally we prove our main result in Section \ref{sec5} using our local result and the transference principle.

\section{Lemmas needed to prove local result}\label{sec2}

\begin{lemma}\label{symmetric-lemma}
For even $n \geq 6$, let $a_0 \geq a_1 \geq \dots \geq a_{n-1}$ be a decreasing sequence of numbers in $[0,1]$. Let $A$ denote their average $A = \frac{1}{n}\sum_{j=0}^{n-1}a_j$. Suppose that for all triples $(i,j,k)$ with $0 \leq i,j,k \leq n-1$ and with $i + j + k \geq n$ we have $$a_ia_j + a_ja_k + a_ka_i \leq \frac{1}{2}(a_1 + a_2 + a_3).$$ Then the average value $A$ is bounded by $\frac{1}{2}.$
\end{lemma}

\begin{proof}
Using the transformation $x_i = 4a_i - 1$, we have a new decreasing sequence $x_0 \geq x_1 \geq \dots \geq x_{n-1}$ in $[-1,3]$. Let $X$ denote the average value $X = \frac{1}{n}\sum_{j=0}^{n-1}x_j$, and observe that it suffices to show $X \leq 1.$ Moreover for all triples $(i,j,k)$ with $0 \leq i,j,k \leq n-1$ and with $i + j + k \geq n$ we have \begin{equation}
    x_ix_j + x_jx_k + x_kx_i \leq 3. \label{sym-x}
\end{equation} 
We first treat the case $n = 6$. We want to show $$x_0 + x_1 + x_2 + x_3 + x_4 + x_5 \leq 6.$$ We may assume that $x_0 > 1$, since otherwise $X \leq x_0 \leq 1$ and we are done. Hence $x_0 + x_i > 0$ for each $i = 0,\dots,n-1.$ Observe we have the following inequalities, using the relations associated to the triples of indices. 
\begin{align*}
    x_2 &\leq \min(1,\frac{3 - x_0x_4}{x_0 + x_4}) \quad \textrm{using} \quad (2,2,2),(0,2,4). \\
    x_3 &\leq \sqrt{3 + x_0^2} - x_0  \quad \textrm{using} \quad (0,3,3). \\
    x_5 &\leq \frac{3 - x_0x_1}{x_0 + x_1} \quad \textrm{using} \quad (0,1,5).
\end{align*}

Hence letting $x = x_0, y = x_1, z = x_4$ it suffices to show $g(x,y,z) \leq 6$ where $$g(x,y,z) = y + z + \min(1,\frac{3-xz}{x+z}) + \frac{3 - xy}{x+y} + (3+x^2)^{\frac{1}{2}},$$ for $x \geq y \geq z$ belonging to $[-1,3]$. Observe additionally that $z$ satisfies $$z \leq 
\min((3+x^2)^{\frac{1}{2}} - x, \frac{3-y^2}{2y}),$$ due to the triples $(0,4,4),(1,1,4).$ 
Now we split into cases based on the range of values taken by $x,y,z$. Note without loss of generality we may assume $x \geq 1$ and $y \geq 0.6.$ Observe on each of the following regions we can bound $g(x,y,z)$ above by the corresponding expressions involving $x$ and $y$.
\begin{align*}
    g(x,y,z) &\leq 1 + 2y + \sqrt{3 + x^2} + \frac{3 - xy}{x+y}, \quad  1 \leq x \leq 2, 0.6 \leq y \leq 1. \\
     g(x,y,z) &\leq 1 + y + 2\sqrt{3 + x^2} - x + \frac{3 - xy}{x+y}, \quad 2 \leq x \leq 3, 0.6 \leq y \leq 1.\\
     g(x,y,z) &\leq  1 + y + \sqrt{3 + x^2} + \frac{3 - xy}{x+y} + \frac{3 - y^2}{2y}, \quad 1 \leq x \leq 2, 1 \leq y \leq x.\\
    g(x,y,z) &\leq 1 + y + 2\sqrt{3 + x^2} - x + \frac{3 - xy}{x+y}, \quad 2 \leq x \leq 2.5, 1 \leq y \leq 1.7.\\
     g(x,y,z) &\leq  1 + y + \sqrt{3 + x^2} + \frac{3 - xy}{x+y} + \frac{3- y^2}{2y}, \quad 2 \leq x \leq 2.5, 1.7 \leq y \leq x.\\
     g(x,y,z) &\leq 1 + y + \sqrt{3 + x^2} + \frac{3 - xy}{x+y} + \frac{3- y^2}{2y}, \quad 2.5 \leq x \leq 3, 1.4 \leq y \leq x.\\
     g(x,y,z) &\leq  y + \sqrt{3 + x^2} + \frac{3 - xy}{x+y} + 1.4, \quad 2.5 \leq x \leq 3, 1 \leq y \leq 1.4, \text{ and } z \leq 0.4.\\
     g(x,y,z) &\leq  y + 2\sqrt{3 + x^2} - x + \frac{3 - xy}{x+y} + \frac{3-x(0.4)}{x + 0.4}, \quad 2.5 \leq x \leq 3, 1 \leq y \leq 1.4, z \geq 0.4.
\end{align*}

Let $f(x,y)$ be the expression which bounds $g(x,y,z)$ on each of the eight regions above. One may check by hand that $f(x,y) \leq 6$ on each region by showing there is no critical point in the interior of the region and maximizing along the boundary. Alternatively one may use, say, Wolfram Alpha and verify that $f(x,y) \leq 6$ on each of the eight regions. Hence we are done with the $n = 6$ case of the lemma.

The remaining $n \geq 8$ case follows the argument of Shao in Lemma $2.1$ in \cite{Shao}.

Let $n = 2m$ for $m \geq 4$ and define 

\[
S_0 = \sum_{i = 0}^{m-1}x_i \quad \textrm{and} \quad S_1 = \sum_{i = m}^{n-1}x_i.
\]

Assume $S_0 + S_1 > 2m$ for a contradiction. Consider the quantity $S_0^2 + 2S_0S_1$; the terms in this sum can be grouped into expressions of the form $x_ix_j + x_jx_k + x_kx_i$ with $0 \leq i,j \leq m-1$, and $m \leq k \leq 2m-1$ such that $i + j + k \equiv 0 \pmod{m}.$ Observe that there are $m^2$ such expressions since for each pair $0 \leq i,j \leq m-1$ there is a unique $k$ satisfying $m \leq k \leq 2m-1$ such that $i + j + k \equiv 0 \pmod{m}.$ Therefore $$S_0^2 + 2S_0S_1 = \sum_{i + j + k \equiv 0 \pmod m}x_ix_j +x_jx_k + x_kx_i.$$ Moreover every term, except for the term corresponding to $(i,j,k) = (0,0,m)$, has its indices sum to at least $n$. Applying \eqref{sym-x}, we find that 

\begin{equation}
    S_0^2 + 2S_0S_1 \leq 3(m^2 -1) + x_0^2 + 2x_0x_m \leq 3m^2 + (x_0^2 - x_m^2) \label{S-sym}
\end{equation} where the final inequality follows since $x_m^2 + 2x_0x_m \leq 3$ due to the triple $(m,m,0)$ and the corresponding relation \eqref{sym-x}.

Now observe that relation \eqref{sym-x} corresponding to the triple $(m-1,m-1,m-1)$ yields $x_{m-1}^2 \leq 1.$ Hence $x_{m-1} \leq 1$ and so $S_0 \leq 3m  - 2.$ Combining this with our assumption that $S_0 + S_1 > 2m$ it follows that $S_1 > -m +2.$ 

First suppose $S_1 \leq 0$, and so $$4m^2 < (S_0+S_1)^2 \leq (S_0^2 + 2S_0S_1)  + S_1^2 \leq 3m^2 + 9 + (m-2)^2,$$ where we used $x_0^2 - x_m^2 \leq 9$. This implies $0 < 13 - 4m$ which is a contradiction for $m \geq 4.$ Now let $S_1 > 0.$ It follows that $x_m > 0$ and $mx_m \geq S_1.$ Hence $$4m^2 < (S_0 + S_1)^2 \leq 3m^2 + (x_0^2 - x_m^2) + m^2x_m^2,$$ by \eqref{S-sym}. 

Rearranging we obtain \begin{equation}
    x_0^2 > (m^2-1)(1-x_m^2) +1. \label{l-b x_0^2}
\end{equation} Using the relation \eqref{sym-x} corresponding to the triple $(0,m,m)$ we have the inequality $$x_0x_m \leq \frac{3-x_m^2}{2} \leq 1 + \frac{1-x_m^2}{2}.$$ Squaring we obtain $$(x_0x_m)^2 \leq 1 + (1-x_m^2) + \frac{(1-x_m^2)^2}{4}.$$ Combining this with the lower bound \eqref{l-b x_0^2}, and using $x_m^2 \leq 1$ which follows from the relation \eqref{sym-x} corresponding to the triple $(m,m,m)$, we obtain the inequality $$x_m^2 < \frac{9}{4m^2 - 3}.$$ Hence for $m \geq 4$, $$x_m^2 \leq \frac{9}{4m^2-3}.$$ Therefore by \eqref{l-b x_0^2}, we obtain $$x_0^2 > (m^2 - 1)(1 - \frac{9}{4m^2-3}) + 1.$$ which contradicts $x_0 \leq 3$ for all $m \geq 4$.
\end{proof}
We now state an asymmetric variant of Lemma \ref{symmetric-lemma} which is needed for the proof of Proposition \ref{local-prop1}. The statement is identical to Lemma 2.2 in \cite{Shao} but with $\frac{1}{2}$ replacing $\frac{5}{8}$, and Shao's proof goes through with some relatively minor adjustment. We will nevertheless provide the complete proof here for completeness.

\begin{lemma}\label{asymmetric-lemma}
Let $n \geq 10$ be an even integer and $\{a_i\}$,$\{b_i\}$,$\{c_i\}$ with $0 \leq i \leq n-1$ decreasing sequences in $[0,1]$.  Let $A,B,C$ be the averages of the sequences $a_i,b_i,c_i$ respectively. If for all $(i,j,k)$ with $i+j + k \geq n $ we have $$a_ib_j + b_jc_k + c_ka_i \leq \frac{1}{2}(a_i + b_j + c_k),$$ then $AB + BC + CA \leq \frac{1}{2}(A+B+C)$
\end{lemma}

\begin{proof}
As before we use the change of variables

\[
x_i = 4a_i - 1, \quad y_i = 4b_i - 1 \quad \textrm{and} \quad z_i = 4c_i - 1,
\]
 for $0 \leq i \leq n-1.$ This results in decreasing sequences $\{x_i\}, \{y_i\},\{z_i\}$ in $[-1,3]$, denote their averages by $X,Y,Z$ respectively. For each triple of indices $(i,j,k)$ such that $i + j + k \geq n$ we have \begin{equation}
    x_iy_j + y_jz_k + z_kx_i \leq 3. \label{xyz-relation}
\end{equation} It suffices to show $XY + YZ + ZX \leq 3.$ Let $n = 2m$ and define 

\[
X_0 = \sum_{i=0}^{m-1}x_i, X_1 = \sum_{i=m}^{n-1}x_i , Y_0 = \sum_{i=0}^{m-1}y_i, Y_1 = \sum_{i=m}^{n-1}y_i, Z_0 = \sum_{i=0}^{m-1}z_i, Z_1 = \sum_{i=m}^{n-1}z_i.
\]

Let $\mathcal{M}$ be the set of all triples $(i,j,k)$ with $0 \leq i,j \leq m-1$, $m \leq k \leq n-1$ and such that $i + j + k \equiv 0 \pmod{m}.$

 We now sum the relation \eqref{xyz-relation} over all triples $(i,j,k) \in \mathcal{M}$ except for $(0,0,m)$ which yields the inequality $$\sum_{(i,j,k) \in \mathcal{M}}x_iy_j + y_jz_k + z_kx_i - (x_0y_0 + y_0z_m + z_mx_0) \leq 3(m^2-1).$$ Observe the sum can be expressed as $X_0Y_0 + Y_0Z_1 + Z_1X_0$ and so we have the inequality $$X_0Y_0 + Y_0Z_1 + Z_1X_0 \leq 3(m^2-1) + (x_0y_0 + y_0z_m + z_mx_0).$$ By symmetry we also have the inequalities 
 
 $$X_0Y_1 + Y_1Z_0 + Z_0X_0 \leq 3(m^2-1) + (x_0y_m + y_mz_0 + z_0x_0),$$

 and

 $$X_1Y_0 + Y_0Z_0 + Z_0X_1 \leq 3(m^2-1) + (x_my_0 + y_0z_0 + z_0x_m).$$

Combining these inequalities we have

\begin{multline*}
n^2(XY + YZ + ZX) = (X_0+X_1)(Y_0 +Y_1) + (Y_0 + Y_1)(Z_0 + Z_1)  \\ + (Z_0 + Z_1)(X_0 + X_1)  \leq 9(m^2-1) + U  - W + V,
\end{multline*}
where
\begin{align*}
    U &= (x_0 + x_m)(y_0 + y_m) + (y_0 + y_m)(z_0 + z_m) + (z_0 + z_m)(x_0 + x_m), \\
    V &= X_1Y_1 + Y_1Z_1 + Z_1X_1, \\
    W &= (x_my_m + y_mz_m + z_mx_m).
\end{align*}

Hence our goal becomes to show \begin{equation}
    U + V - W \leq 3m^2 + 9.
\end{equation}
Now observe that summing condition \eqref{xyz-relation} over all triples $(i,j,k)$ such that $m \leq i,j,k \leq n-1$ and $i + j + k \equiv 0 \pmod{m}$, except for $(m,m,m)$, yields the inequality $$X_1Y_1 + Y_1Z_1 + Z_1X_1 - (x_my_m + y_mz_m + z_mx_m) \leq 3(m^2-1),$$ and so $$V - W \leq 3(m^2-1).$$ Now let \[
r = x_0 + x_m, \quad s = y_0 + y_m \quad \textrm{and} \quad t = z_0 + z_m
\]
and so $U = rs + st + tr.$ Observe $r,s,t \in [-2,6]$. Now suppose one of $\{r+s,s+t,t+s\}$, say $r+s$, is negative. Then $$U = rs + t(r+s) \leq rs + (-2)(r+s) = (r-2)(s-2) -4  \leq (-4)(-4) - 4 = 12.$$ Therefore $$U + V - W \leq 3m^2 + 9,$$ as desired. Hence we may assume $r+s,s+t,t+s$ are nonnegative and consequently that $U$ is an increasing function of $r,s,t.$ 
We will now reduce to four cases based upon the parities of $X_1,Y_1,Z_1.$ First, observe that \[
x_m-(m-1) \leq X_1 \leq mx_m, \quad y_m - (m-1) \leq Y_1 \leq my_m \quad \textrm{and} \quad z_m - (m-1) \leq Z_1 \leq mz_m.
\] \\
\textbf{Case 1} Suppose $X_1,Y_1,Z_1 < 0.$ \\ 
In this case $V = X_1Y_1 + Y_1Z_1 + Z_1X_1$ is a decreasing function of $X_1,Y_1,Z_1$.$$V \leq (x_m - (m-1))(y_m - (m-1)) + (y_m-(m-1))(z_m - (m-1)) + (z_m - (m-1))(x_m - (m-1)),$$ and so $$V \leq 3(m-1)^2 + W -2(m-1)(x_m + y_m + z_m).$$ Since $U$ is increasing in $r,s,t$ we can bound it above, \begin{align*}
    U &\leq (3 + x_m)(3+y_m) + (3 + y_m)(3+z_m) + (3 + z_m)(3+x_m) \\
    U & \leq W + 27 + 6(x_m + y_m + z_m).
\end{align*}
Hence \begin{multline*}
    U + V - W \leq (x_my_m + y_mz_m + z_mx_m) + 3(m-1)^2 + 27 - (2m-8)(x_m + y_m + z_m) \\ \leq 3(m-1)^2 + 27 + (x_m - m + 4)(y_m - m + 4) + (y_m - m + 4)(z_m - m +4)  \\ + (z_m - m + 4)(x_m - m + 4) - 3(4-m)^2.
\end{multline*}
Now observe that $$x_m - m + 4,y_m - m + 4,z_m - m + 4 \in [-m + 3, -m + 7],$$ and so the maximum value in the upper bound above is attained when $x_m = y_m = z_m = -1.$ Therefore $U + V - W \leq 3(m-1)^2 + 27 + 3(3-m)^2 - 3(4-m)^2 \leq 3m^2 + 9.$ \\
\textbf{Case 2} Now suppose exactly two of $X_1,Y_1,Z_1$ are negative, say $X_1,Y_1 < 0, Z_1 \geq 0.$ \\ 
Then $$V = X_1Y_1 + Y_1Z_1 + Z_1X_1 \leq X_1Y_1 \leq m^2$$ since $-m \leq x_m - (m-1) \leq X_1 \leq 0$ and similarly for $Y_1.$ Now observe that
\begin{multline*}
    U-W = (x_0y_m + y_mz_m + z_mx_0 ) + (x_my_0 + y_0z_m + z_mx_m) +  (x_my_m + y_mz_0 + z_0x_m) \\ 
    + (x_0y_0 + y_0z_0 + z_0x_0) - W  \leq 3 + 3 + 3 + 27 + 9 \leq 45.
\end{multline*}
Here we use relation \eqref{xyz-relation} to estimate the first three parentheses, and the remaining terms we estimate trivially.
Therefore $$U + V - W \leq m^2 + 45 \leq 3m^2 + 9$$ for $m \geq 5.$ \\
\textbf{Case 3} Now suppose exactly one of $X_1,Y_1,Z_1$ is negative, say $X_1 < 0$, and that at least one of $X_1 + Y_1$ or $X_1 + Z_1$ is negative, say $X_1 + Y_1 < 0$. \\ 
Then $V = X_1Y_1 + Y_1Z_1 + Z_1X_1 = Z_1(X_1 + Y_1) + X_1Y_1 \leq 0.$ We can bound $U - W \leq 45$ as in Case 2, and so $U + V - W \leq 45 \leq 3m^2 + 9$ for $m \geq 5.$ \\
\textbf{Case 4} Now suppose $X_1 + Y_1, Y_1 + Z_1, Z_1 + X_1$ are all nonnegative.\\ 
It follows that $x_m + y_m, y_m + z_m, z_m + x_m$ are all nonnegative. Therefore $V = X_1Y_1 + Y_1Z_1 + Z_1X_1$ is an increasing function in $X_1,Y_1,Z_1.$ Hence \begin{equation} \label{Vbound}
    V \leq m^2(x_my_m + y_mz_m + z_mx_m) = m^2W 
\end{equation}

As in Case 2, expanding $U$ and using relation \eqref{xyz-relation} we have the bound \begin{equation} \label{Ubound}
    U \leq 9 + (x_0y_0 + y_0z_0 + z_0x_0).
\end{equation}
Subtracting $x_0y_0 + y_0z_0 + z_0x_0$ from both sides, we obtain \begin{equation}
    W + x_0(y_m + z_m) + y_0(x_m + z_m) + z_0(x_m + y_m) \leq 9. \label{U - 00 ineq}
\end{equation}
Now suppose that $x_m + y_m, y_m + z_m, z_m + x_m \geq 0.6.$ This implies that $$x_0 + y_0 - 10(x_m+y_m) \leq 6 - 10(0.6) = 0,$$ and so $(x_0 + y_0 - 10(x_m + y_m))(z_0 - z_m) \leq 0.$ Expanding we obtain $$(x_0z_0 + y_0z_0)+10(x_mz_m + y_mz_m) \leq x_0z_m + y_0z_m + 10(x_mz_0 + y_mz_0).$$ By symmetry we also have the inequalities

$$(x_0y_0 + z_0y_0)+10(x_my_m + z_my_m) \leq x_0y_m + z_0y_m + 10(x_my_0 + z_my_0),$$ and

$$(z_0x_0 + y_0x_0)+10(z_mx_m + y_mx_m) \leq z_0x_m + y_0x_m + 10(z_mx_0 + y_mx_0).$$
Adding these inequalities we obtain \begin{equation*}
    2(x_0y_0 + y_0z_0 + z_0x_0) + 20W \leq 11(z_0(x_m + y_m) + y_0(x_m + z_m) + x_0(y_m + z_m)).
\end{equation*}
Inserting the bound \eqref{U - 00 ineq}, we obtain $$ 2(x_0y_0 + y_0z_0 + z_0x_0) + 20W \leq 11(9-W),$$
and so we have
\begin{equation} \label{Wbound}
        2(x_0y_0 + y_0z_0 + z_0x_0) + 31W \leq 99.
\end{equation}
Therefore, using \eqref{Vbound}, \eqref{Vbound} and \eqref{Wbound}, we obtain \begin{align*}
    U + V - W &\leq 9 + (x_0y_0 + y_0z_0 + z_0x_0) + m^2W - W \\
    & \leq 9 + \frac{1}{2}(99-31W) + m^2W - W \\
    & \leq 49.5 + W(m^2 - 16.5) \\
    & \leq 3m^2, 
\end{align*}
Where the last inequality holds since $W \leq 3$, and $m \geq 5.$
Finally we suppose at least one of $x_m + y_m, y_m + z_m, z_m + x_m$ is less than $0.6$, say $x_m + y_m < 0.6.$ Since $x_m + y_m \geq 0$, we have $x_my_m < (\frac{0.6}{2})^2$. Observe $$W = z_m(x_m + y_m) + x_my_m \leq 3(0.6) + (\frac{0.6}{2})^2 < 2.$$ Hence, using this bound on $W$, \eqref{Ubound} and \eqref{Vbound}, $$U + V - W \leq 9 + (x_0y_0 + y_0z_0 + z_0x_0) + 2(m^2 - 1) \leq 36 + 2(m^2-1) \leq 3m^2 + 9$$ for $m \geq 5.$
\end{proof}

\section{Proof of local result}\label{sec3}
In this section we will prove our local result Theorem  \ref{local-result}. It is our key tool which will make the transference principle argument work, and it will follow from Propositions \ref{local-prop1} and \ref{local-prop2}.

\begin{proposition}\label{local-prop1}
Let $m$ be a squarefree integer which satisfies $(m,30) = 1.$ Let $f: \Z_m^* \rightarrow [0,1]$ satisfy $$\sum_{a \in \Z_m^*}f(a) > \frac{\varphi(m)}{2}.$$ Then for any $x \in \Z_m$ there exists $a_1,a_2,a_3 \in \Z_m^*$ with $x = a_1 + a_2 + a_3$ such that $$f(a_1)f(a_2) + f(a_2)f(a_3) + f(a_3)f(a_1) > \frac{1}{2}(f(a_1) + f(a_2) + f(a_3)).$$
\end{proposition} 

\begin{proof}See Proposition $3.1$ in Shao \cite{Shao}. The proof is identical upon swapping the occurrences of $\frac{5}{8}$ with $\frac{1}{2}$.
\end{proof}

\begin{proposition}\label{local-prop2}
Let $f_1,f_2,f_3 : \Z_{15}^* \rightarrow [0,1]$ be functions such that $f_i(x) = 0$ for all $x \equiv 2 \pmod{3}$, $i = 1,2,3$. Let $$F_i = \sum_{a \in \Z_{15}^*}f_i(a),$$ and suppose that \begin{equation} \label{F1F2F3}
    F_1F_2 + F_2F_3 + F_3F_1 > 2(F_1 + F_2 + F_3).
\end{equation} Then for any $x \equiv 0 \pmod{3}$ there exist $a_1,a_2,a_3 \in \Z_{15}^*$ with $x = a_1 + a_2 + a_3$ such that $$f_i(a_i) > 0, f_1(a_1) + f_2(a_2) + f_3(a_3) > \frac{3}{2}.$$
\end{proposition}

\begin{proof}
Let $A_i$ be the support of $f_i$, and so $A_i = \{x \in \Z_{15}^* : f_i(x) > 0 \} \subset \{1,4,7,13\}.$ Let $n_i = |A_i|$, without loss of generality we may assume that $n_1 \geq n_2 \geq n_3.$ Since $n_i \geq F_i$, we have that \begin{equation} \label{n1n2n3ineq}
    n_1n_2 + n_2n_3 + n_3n_1 > 2(n_1 + n_2 + n_3).
\end{equation} Let $M$ be the set of all possible $(n_1,n_2,n_3)$ satisfying \eqref{n1n2n3ineq}, with $n_1,n_2,n_3 \in \{0,1,2,3,4\}.$ Fix some $(n_1,n_2,n_3) \in M.$ For $i \in \{1,2,3\}$, write \[A_i = \{x_{i1}, x_{i2}, \dots, x_{in_i}\} \quad \textrm{and} \quad y_{ij} = f_i(x_{ij}),
\] for $1 \leq j \leq n_i.$ Without loss of generality assume that $y_{i1} \geq y_{i2} \geq \dots \geq y_{in_i}.$ 

Let $J$ be the set of all triples $(j_1,j_2,j_3)$ satisfying \begin{equation} \label{j1j2j3}
    j_1 + j_2 + j_3 > 6, \quad 1 \leq j_i \leq n_i.
\end{equation}

Suppose $y_{1j_1} + y_{2j_2} + y_{3j_3} > \frac{3}{2}$ for some $(j_1,j_2,j_3) \in J$. The three sets $A_i = \{x_{i1},\dots, x_{ij_i}\}$ for $i = 1,2,3$ are subsets of $$\Z_{15}^* \cap \{x \in \Z_{15} : x \equiv 1 \pmod 3\}.$$ Making use of the isomorphism $\Z_{15}^* \cong \Z_3^* \times \Z_5^*$ we can view $A_1,A_2,A_3 \subset \Z_5^*$. It follows by Cauchy-Davenport-Chowla that $$|A_1+A_2+A_3| \geq \min(5,|A_1|+|A_2|+|A_3|-2) \geq 5,$$ where the last inequality follows since $|A_1|+|A_2|+|A_3| = j_1 + j_2 + j_3 > 6$ by \eqref{j1j2j3}. Therefore viewing $A_1,A_2,A_3$ as subsets of $$\Z_{15}^* \cap \{x \in \Z_{15} : x \equiv 1 \pmod 3\},$$ it follows that $$A_1 + A_2 + A_3 = \{x\in \Z_{15}: x\equiv 0 \pmod{3}\}.$$ Hence for any $x \equiv 0 \pmod{3}$, we can write $x \equiv a_1 + a_2 + a_3 \pmod{15}$ with $a_i \in A_i,$  such that $$f(a_i) \geq y_{ij_i} > 0, \quad f(a_1) + f(a_2) + f(a_3) \geq y_{1j_1} + y_{2j_2} + y_{3j_3} > \frac{3}{2}.$$

We now suppose for a contradiction that $$y_{1j_1} + y_{2j_2} + y_{3j_3} \leq \frac{3}{2}$$ for each triple $(j_1,j_2,j_3) \in J.$ We can set up the following optimization problem with variables $y_{ij} \in [0,1]$ for $1 \leq j \leq n_i$ and constraints given by  $y_{1j_1} + y_{2j_2} + y_{3j_3} \leq \frac{3}{2}$ for all $(j_1,j_2,j_3) \in J.$ Our objective is to maximize the sum of all variables: \begin{equation} \label{S opt}
    S = \sum_{i=1}^3 \sum_{j=1}^{n_i}y_{ij} = F_1 + F_2 + F_3.
\end{equation} The constraints and the objective function in this optimization problem are all linear, and the maximum of $S$ can be found using a linear programming algorithm. Our conclusion is the following.

$$F_1 + F_2 + F_3 \leq
\begin{cases}
 6.5      &     (n_1,n_2,n_3) = (4,4,1) \\
 6.2      & (n_1,n_2,n_3) = (4,4,2)\\
 6        &     \textrm{otherwise}

\end{cases}
$$
Suppose $F_1 + F_2 + F_3 \leq 6.$ Applying Cauchy-Schwarz it follows that \begin{equation} \label{F1F1F3CONT}
    F_1F_2 + F_2F_3 + F_3F_1 \leq \frac{1}{3}(F_1 + F_2 + F_3)^2 \leq 2(F_1 + F_2 + F_3),
\end{equation} which contradicts the condition \eqref{F1F2F3}. We will now show the other possibilities for the triple $(n_1,n_2,n_3)$ also lead to a contradiction.

For notational convenience, we will write $$T(x,y,z) = xy + yz + zx - 2(x+y+z).$$ By the condition \eqref{F1F2F3} it follows that $T(F_1,F_2,F_3) > 0$ and that $T(x,y,z)$ is an increasing function of $x,y,z$ in the range $x \geq F_1$, $y \geq F_2$, $z \geq F_3$. 

Now suppose $(n_1,n_2,n_3) = (4,4,2)$. If $F_3 \leq 1,$ it follows that $$T(F_1,F_2,F_3) \leq T(2.6,2.6,1) < 0.$$ Hence we may assume $F_3 \geq 1,$ solving the optimization problem as before but with this additional constraint, we have $$F_1 + F_2 + F_3 \leq 6.$$ This is a contradiction due to inequality \eqref{F1F1F3CONT}.

Finally suppose $(n_1,n_2,n_3) = (4,4,1),$ then $$F_1 + F_2 + F_3 \leq 6.5.$$ Consider the case where $F_3 \leq 0.5,$ then $$T(F_1,F_2,F_3) \leq T(3,3,0.5) < 0.$$ Hence we may assume $F_3 \geq 0.5,$ including this as an additional constraint and solving the optimization problem as before we obtain $$F_1 + F_2 + F_3 \leq 6.$$ This is a contradiction due to inequality \eqref{F1F1F3CONT}.
\end{proof}

We now prove our local result by an induction argument \ref{local-result}. 

\begin{proof}
Let $m$ be an odd square free integer divisible by $3$. Let $m = m_1m_2$ where $m_2$ divides $15$ and $(m_1,30) = 1$. Note that $m_2 \in \{3,15\}$. Identifying $\Z_m \equiv \Z_{m_1} \times \Z_{m_2}$, let $(u,v) \in \Z_m$ where $u \in \Z_{m_1}, v \in \Z_{m_2}$ with $v \equiv 0 \pmod{3}$. Define $$f'(x) = \frac{2}{\phi(m_2)}\sum_{y \in \Z_{m_2}^*}f(x,y)$$ for $x \in \Z_{m_1}^*$. Observe that 
$$\sum_{x \in \Z_{m_1}^*}f'(x) > \frac{\phi(m_1)}{2},
$$
and so we may use Proposition \ref{local-prop1} to conclude that there exist $a_1,a_2,a_3 \in \Z_{m_1}^*$ with $u = a_1 + a_2 + a_3$ such that \begin{equation} \label{f' sym}
    f'(a_1)f'(a_2) + f'(a_2)f'(a_3) + f'(a_3)f'(a_1) > \frac{1}{2}(f'(a_1) + f'(a_2) + f'(a_3)).
\end{equation} 
Suppose first that $m_2 = 3$. In this case we have $f'(a_i) = f(a_i,1)$ for $i = 1,2,3$. Setting $x_i = f(a_i,1)$ for $i = 1 ,2,3$, Inequality \eqref{f' sym} becomes
\begin{equation}\label{f'sym-x}
 x_1x_2 + x_2x_3 + x_3x_1 > \frac{1}{2}(x_1+x_2+x_3).   
\end{equation}
We may equivalently express Inequality \ref{f'sym-x} as 
\begin{equation}\label{sym-inequality-mod3}
    (x_1+x_2+x_3)^2 - (x_1+x_2+x_3) > x_1^2+x_2^2+x_3^3.
\end{equation}
Now observe, by an application of Cauchy-Schwarz, we have that 
\begin{equation}\label{c-s-mod3}
    x_1^2 + x_2^2 + x_3^2 \geq \frac{(x_1+x_2+x_3)^2}{3}.
\end{equation}
Combining Inequalities \eqref{sym-inequality-mod3} and \eqref{c-s-mod3} it follows that
$$
x_1 + x_2 + x_3 > 3/2.
$$
Hence $(u,v) = (u,0) = (a_1,1) + (a_2,1) + (a_3,1)$ with $$f(a_1,1) + f(a_2,1) + f(a_3,1) > 3/2.$$ Since Inequality \ref{f'sym-x} implies that $x_i > 0$, it follows that $f(a_i,1) > 0$ for $i = 1,2,3$ as desired.

Now suppose $m_2 = 15$. Define $f_i : \Z_{15}^* \rightarrow [0,1]$ for $i = 1,2,3$ by $$f_i(y) = f(a_i,y)$$ for each $y \in \Z_{15}^*.$ 
The condition \eqref{f' sym} satisfied by the $f'(a_i)$ allows the use of Proposition \ref{local-prop2} to conclude that there exists $b_1,b_2,b_3 \in \Z_{15}^*$ with $v = b_1 + b_2 + b_3,$ such that \[f_i(b_i) > 0 \quad \textrm{and} \quad f_1(b_1) + f_2(b_2) + f_3(b_3) > \frac{3}{2}.\] Therefore we have $(u,v) = (a_1,b_1) + (a_2,b_2) + (a_3,b_3)$ with \[f(a_i,b_i) > 0 \quad \textrm{and} \quad f(a_1,b_1) + f(a_2,b_2) + f(a_3,b_3) > \frac{3}{2},\] as desired.
    
\end{proof}

\section{Transference principle }\label{sec4}
In this section we state the variant of the transference principle which we will use in the next section to deduce our main result. We will record the statement as it is in Proposition 4.1 in \cite{Shao}.

Let $\Z_N$ be the cyclic group of integers modulo $N.$. Define the Fourier transform of a function $f: \Z_N\rightarrow \C$ by
$$ \wh{f}(r) = \sum_{n \in \Z_N} f(n) e_N(rn), \quad r \in \Z_N,$$ where $e_N(rn) = \exp(2\pi rn/N ).$ 
For two functions $f_1,f_2:\Z_N\rightarrow\C$, their convolution $f_1*f_2$ is defined by
$$ f_1*f_2(n) = \sum_{n_1 \in \Z_N} f_1(n_1) f_2(n-n_1), \quad n \in \Z_N.$$

\begin{proposition}\label{transference}
Let $N$ be a sufficiently large prime. Suppose that $\nu_i: \Z_N \rightarrow \R^+$ and $f_i : \Z_N \rightarrow \R^+ (i = 1,2,3)$ are functions satisfying the majorization condition \begin{equation} \label{majorant}
    0 \leq f_i(n) \leq \nu_i(n),
\end{equation}
and the mean condition \begin{equation} \label{mean}
    \min(\delta_1,\delta_2,\delta_3,\delta_1 +\delta_2 + \delta_3-1) \geq \delta
\end{equation}

for some $\delta > 0,$ where the $\delta_i = \sum_{x \in \Z_N}f_i(x).$ Suppose that $\nu_i$ and $f_i$ also satisfy the pseudorandomness conditions \begin{equation} \label{decay}
    |\hat{\nu}_i(r) - \delta_{r,0}| \leq \eta 
\end{equation} for all $r \in \Z_N$, where $\delta_{r,0}$ is the Kronecker delta, and \begin{equation} \label{meanvalue}
    \|\hat{f}_i\|_q = \Big(\sum_{r \in \Z_N}|\hat{f}_i(r)|^q\Big)^{\frac{1}{q}} \leq M
\end{equation}
for some $2 < q < 3$ and $\eta, M > 0.$ Then for any $x \in \Z_N,$ we have \begin{equation}
    \sum_{y,z \in \Z_N}f_1(y)f_2(z)f_3(x-y-z) \geq \frac{c(\delta)}{N}
\end{equation}
for some constant $c(\delta) > 0$ depending only on $\delta,$ provided that $\eta = \eta(\delta,M,q)$ is small enough.

\end{proposition}

\begin{proof}
    See Proposition 4.1 in \cite{Shao}.
\end{proof}

\section{Ternary Goldbach for dense subsets of primes in a fixed residue class}\label{sec5}

In this section we prove Theorem \ref{density-prime-1/2}.

\begin{proof}
Let $n$ be a sufficiently large odd integer such that $n \equiv 0 \pmod 3$. Let $A \subset \CP_{1,3}$ be a subset satisfying $\vdelta(A) > 1/2$. Then there exists a positive constant $\delta > 0$ such that 

\begin{equation}\label{density hypothesis} |A \cap [1,N]| > (\frac{1}{2} + \delta)\frac{N}{2\log N},
\end{equation} for all sufficiently large positive integers $N.$

 Set $W = \prod_{p \leq z}p$, where $z = z(\delta)$ is a constant sufficiently large in terms of $\delta$. 
For each reduced residue class $b \pmod{W}$, define $$f(b) = \max(\frac{\phi(W)3}{2n}\sum_{\substack{x \in A \cap [1,\frac{2n}{3}] \\ x \equiv b \pmod{W}}} \log x - \frac{\delta}{8},0).$$ Note the support of $f$ is contained in those reduced residue classes congruent to $1 \pmod{3}$ and $0 \leq f(b) \leq 1$ for all $b \in \Z_W^*.$ Moreover by the prime number theorem in arithmetic progressions, \eqref{density hypothesis} implies $$\sum_{b \in \Z_W^*}f(b) > \frac{\phi(W)}{4}.$$ Let $m = W/2$, and since $\Z_W^*$ can be identified with $\Z_m^*$ we may consider $f$ as a function $f : \Z_m^* \rightarrow [0,1]$ satisfying $$\sum_{b \in \Z_m^*}f(b) > \frac{\phi(m)}{4}.$$ Hence we may apply our local result \ref{local-result} to conclude there exist $b_1,b_2,b_3 \in \Z_W^*$ such that $n \equiv b_1 + b_2 + b_3 \pmod{m}$ and \begin{equation}
    f(b_i) > 0, \quad  f(b_1) + f(b_2) + f(b_3) > \frac{3}{2}.
\end{equation} Since $n$ is odd, $n \equiv b_1 + b_2 + b_3 \pmod{2}$ and therefore $n \equiv b_1 + b_2 + b_3 \pmod{W}.$ Now define $$A_i = \{\frac{x - b_i}{W} : x \in A \cap [1,\frac{2n}{3}], x \equiv b_i \pmod W    \},$$ and so it suffices to show that $$\frac{n - b_1 - b_2 - b_3}{W} \in A_1 + A_2 + A_3.$$ Let $\kappa > 0$ be sufficiently small, and choose a prime  $$N \in \Big[\frac{(1 + \kappa)n}{W},\frac{(1 + 2\kappa)n}{W}\Big].$$ By our choice of $N$ it suffices to show that $\frac{n - b_1 - b_2 - b_3}{W} \in A_1 + A_2 + A_3$ when the sets $A_1,A_2,A_3$ are considered subsets of $\Z_N.$

For $i \in \{1,2,3\}$, define $f_i, \nu_i : \Z_N \rightarrow \R_{\geq 0}$ (naturally identifying $\Z_N$ with $\{1,2,\cdots,N\}$) by
$$ \nu_i(n) =
\begin{cases}
             \frac{\phi(W)}{WN}\log(Wn + b_i) &\text{if}\ Wn + b_i\in \CP, \\
 
     0 &\text{otherwise.}
\end{cases} $$
and
$$ f_i(n) = \nu_i(n)1_{A_i}(n). $$
We will show that $f_i,\nu_i$ satisfy the assumptions of Proposition \ref{transference}. Clearly $0 \leq f_i(n) \leq \nu_i(n)$ for every $n$ and so \eqref{majorant} holds trivially. Now we verify the mean value estimate \eqref{meanvalue} on $\|\widehat{f_i}\|_p.$
\begin{lemma}[Mean value estimate]
Suppose that $p > 2$. Then there is an absolute constant $C(p) > 0$ such that  $\|\widehat{f_i}\|_p \leq C(p).$
\end{lemma}
\begin{proof}
See Lemma 6.6 in \cite{Green} with exponent $p=3$ (say)
\end{proof}
Now we verify the the Fourier decay property \eqref{decay} of $\nu_i$.
\begin{lemma}[Fourier decay property]
Suppose that $N$ and $z$ are sufficiently large. Then $$\sup_{r \neq 0}|\hat{\nu_i}(r)| \leq \frac{2\log\log z}{z}.$$
\end{lemma}
\begin{proof}
    See Lemma 6.2 in \cite{Green}.
\end{proof}
This establishes \eqref{decay} when $r \neq 0$. The case when $r = 0$ follows from the prime number theorem.

Now we verify the requirement \eqref{mean} on the averages of $f_i$, $i = 1,2,3$. Observe that 
$$ \delta_i = \sum_{n \in \Z_N}f_i(n) = \frac{\phi(W)}{NW}\sum_{\substack{p \in A \cap [1,\frac{2n}{3}] \\ p \equiv b_i \pmod W}} \log{p} = \frac{2n}{3WN}(f(b_i) + \frac{\delta}{8}) \geq \frac{2f(b_i)}{3} + \frac{\delta}{20}.$$
Hence, we may apply Proposition \ref{transference} to the functions $f_i,\nu_i$ to  conclude that, for sufficiently large $z = z(\delta)$,
$$ \sum_{y,z \in \Z_N}f_1(y)f_2(z)f_3(n-y-z) \geq \frac{c(\delta)}{N}$$
for all values of $n \in \Z_N$. Therefore $A_1 + A_2 + A_3 = \Z_N$ and so we are done.
    
\end{proof}

\subsection*{Acknowledgements}
I would like to thank my advisor Xuancheng Shao and Laurence Wijaya for helpful discussions.

\bibliographystyle{plain}
\bibliography{biblio}

\end{document}